\documentclass[12pt]{amsart}

\usepackage{hyperref}
\usepackage{mathrsfs, euscript}
\usepackage{bbm}
\usepackage{txfonts}
\usepackage{amscd,graphicx}
\usepackage{amsfonts,latexsym,amsmath,amsthm,amsxtra,amssymb,latexsym}

\usepackage{fourier}
\usepackage[scaled]{berasans}

\usepackage{xcolor}

\usepackage{mathtools}
\usepackage{todonotes}
\usepackage[marginratio=1:1,height=9.0in,width=6.5in,tmargin=1.0in]{geometry}
\usepackage[normalem]{ulem}
\usepackage{enumerate,enumitem}

    \newtheorem{theorem}{Theorem}[section]
    \newtheorem{corollary}[theorem]{Corollary}
    \newtheorem{lemma}[theorem]{Lemma}  
    \newtheorem{prop}[theorem]{Proposition}

    \newtheorem{defn}[theorem]{Definition}

    \newcommand{\Poincare}{Poincar\'{e}~}
    
    \numberwithin{equation}{section}

\theoremstyle{remark}
\newtheorem{remark}{Remark}
\newtheorem{remarks}[remark]{Remarks}

\newcommand{\bb}{\mathbb}
\newcommand{\mf}{\mathfrak}

\newcommand{\scr}{\mathscr}

\newcommand{\la}{\lambda}
\newcommand{\La}{\Lambda}
\newcommand{\ga}{\gamma}
\newcommand{\eps}{\epsilon}
\newcommand{\Ga}{\Gamma}

\newcommand{\floor}[1]{\left\lfloor #1 \right\rfloor}
\newcommand\rest[1]{\raisebox{-.5ex}{$|$}_{#1}}
\newcommand\restr[2]{{\left.\kern-\nulldelimiterspace #1 \vphantom{\big|} \right|_{#2} }}

\DeclareMathOperator{\vol}{vol}

\newcommand{\op}{\operatorname}
\newcommand{\abs}[1]{\left\lvert \vcenter{\hbox{$\displaystyle #1 $}} \right\rvert}
\newcommand{\norm}[1]{\left\lVert \vcenter{\hbox{$\displaystyle #1 $}} \right\rVert}
\newcommand{\inner}[1]{\left\langle #1 \right\rangle}

\newcommand{\set}[1]{\left\{ #1 \right\} }

\newcommand{\grad}{\nabla}
\newcommand{\til}{\widetilde}

\DeclareMathOperator{\Isom}{Isom}

\DeclareMathOperator{\comass}{comass}
\DeclareMathOperator{\mass}{mass}
\DeclareMathOperator{\Inj}{Inj}
\DeclareMathOperator{\Ric}{Ric}

\newcommand{\R}{{\bb R}}
\newcommand{\Z}{{\bb Z}}
\newcommand{\N}{{\bb N}}

\newcommand{\HH}{{\bb H}}

\renewcommand{\bar}{\overline}

\def\be#1\ee{\begin{align}\begin{split} #1 \end{split}\end{align}}
\def\beq#1\eeq{\begin{align*}\begin{split} #1 \end{split}\end{align*}}

\begin{document}

\title{Homological norms on nonpositively curved manifolds}

\author{Chris Connell and Shi Wang}

\address{Department of Mathematics,
Indiana University, Bloomington, IN 47405}
\email{connell@indiana.edu}
\address{Institute of Mathematical Sciences,
ShanghaiTech University, Shanghai, China}
\email{shiwang.math@gmail.com}

\subjclass[2010]{Primary 53C23; Secondary 53C20, 58J35}

\begin{abstract}
We relate the Gromov norm on homology classes to the harmonic norm on the dual cohomology and obtain double sided bounds in terms of the volume and other geometric quantities of the underlying manifold. Along the way, we provide comparisons to other related norms and quantities as well. 
\end{abstract}

\maketitle

\section{Introduction and Results}

The vector space $C_p(X, \mathbb R)$ of singular $p$-chains of a topological
space $X$ comes equipped with a natural choice of basis consisting of the set of
all continuous maps from the $p$-dimensional Euclidean simplex into $X$. The
$\ell^1$-norm on $C_p(X, \mathbb R)$ associated to this basis descends to a
semi-norm $\norm{\cdot }_1$ on the singular homology $H_p(X, \R)$ by
taking the infimum of the norm within each equivalence class, or more precisely,

\begin{defn} The Gromov norm of a $p$-homology class $\alpha$ is
	$$\norm{\alpha}_1:=\inf \set{\sum_i \abs{a_i}\,:\, \left[\sum_i a_i \sigma_i\right]=\alpha\in H_p(X,\R) }.$$
\end{defn} 

This norm was
introduced by Thurston (\cite[Chapter 6]{Thurston77}), and later generalized by
Gromov (\cite{Gromov82}), and an important special case is the {\em simplicial volume,} written $\norm{M}$, of
a closed oriented connected $n$-manifold $M$ which is defined to be $\norm{[M]}_1$, where $[M]$
is the fundamental class in $H_n(M, \R)$--the image of a preferred generator in $H_n(M, \Z)$ under the change of coefficient map. More concretely,
\[
\norm{M}=\inf \set{\sum_i \abs{a_i}\,:\, \left[\sum_i a_i \sigma_i\right]=[M]\in H_n(M,\R) },
\]
where the infimum is taken over all singular cycles with real coefficients
representing the fundamental class in the top homology group of $M$.  This
invariant is multiplicative under finite covers, so the definition can be extended to
closed connected non-orientable manifolds as well. In what follows, we will always assume that our manifolds are connected.

For a cohomology class $\beta\in H^p(M,\R)$ we denote the \emph{harmonic norm} by $\norm{\beta}_H$ and define it to be the $L^2$-norm of the unique harmonic form representing the class $\beta$ in the DeRham cohomology. Since the harmonic representative minimizes the $L^2$-norm among all forms, we have
\begin{defn} The harmonic norm of a cohomology class $\beta\in H^p(M,\R)$ is
	$$\norm{\beta}_H:=\inf_{[\eta]=\beta}\norm{\eta}_2:=\inf_{[\eta]=\beta}\left(\int_M \eta\wedge*\eta\right)^{\frac12},$$
	where the infimum is taken over all closed differential forms $\eta$ representing $\beta$. The infimum is achieved by the harmonic representative $\omega$ satisfying $\Delta \omega=0$, i.e.  $\norm{\omega}_2=\norm{\beta}_H$. 
\end{defn}

For $M$ a closed hyperbolic 3-manifold, Bergeron, \c{S}eg\"{u}n and Venkatesh established the following result,
\begin{theorem}[Proposition 4.2 of \cite{Bergeron-Sengun-Venkatesh}]
For any closed orientable hyperbolic three manifold $M_0$ there are constants $C_1,C_2$ depending only on $M_0$ such that for any finite cover $M$ of $M_0$ and all $\beta\in H^1(M,\R)$ with \Poincare dual $\beta^*\in H_2(M,\R)$, we have
\[
\frac{C_1}{\vol(M)}\norm{\beta^*}_1\leq \norm{\beta}_H\leq C_2\norm{\beta^*}_1
\]
\end{theorem}

In the same setting, Brock and Dunfield significantly tightened the relationship between these two norms.

\begin{theorem}[Theorem 1.2 of \cite{Brock-Dunfield}]\label{thm:B-D}
For all closed orientable hyperbolic three manifolds $M$ and all $\beta\in H^1(M,\R)$ with \Poincare dual $\beta^*\in H_2(M,\R)$, we have
\[
\frac{\pi}{2\sqrt{\vol(M)}}\norm{\beta^*}_1\leq \norm{\beta}_H\leq \frac{5 \pi}{\sqrt{\Inj(M)}}\norm{\beta^*}_1
\]
\end{theorem}

\begin{remark}
	The above theorems were originally stated in terms of the Thurston norm $\norm{\beta}_{Th}$, which by a result (\cite[Corollary 6.18]{Gabai}) of Gabai  is just half the Gromov norm $\norm{\beta^*}_1$.
\end{remark}

In the case of Theorem \ref{thm:B-D}, the same authors showed through examples that their bounds are optimal up to a multiplicative constant on the lower side and optimal up to a sublinear function of $\Inj(M)$ on the upper side. In particular they provide an example of a sequence of hyperbolic 3-manifolds $M_i$ with uniformly bounded volumes and $\Inj(M_i)\to 0$ and $\beta_i\in H^1(M_i,\R)$ with $\frac{\norm{\beta_i}_H}{\norm{\beta_i^*}_1}\geq C \sqrt{-\log \Inj(M_i)}$ for some constant $C>0$.

The main purpose of our paper is to extend these results to all dimensions and to well-known classes of nonpositively curved manifolds. In fact, the lower bound does not even require nonpositive curvature, and it is a consequence of the duality principle and Gromov's comass inequality.

\begin{theorem}\label{thm:upper-bound}
	Let $M$ be a closed oriented manifold of dimension $n$ normalized so that $\op{Ric}\geq -(n-1)$. If $\beta\in H^p(M,\R)$ is any cohomology class, then the Gromov norm of its \Poincare dual has an upper bound
	$$\norm{\beta^*}_1\leq (n-p)!(n-1)^{n-p}\sqrt{\vol(M)}\,\norm{\beta}_H.$$
\end{theorem}

\begin{remark}
	Distinct from the \Poincare dual, we may define a {\em pairing dual} $\hat{\beta}\in H_p(M,\R)$ of a class $\beta\in H^p(M,\R)$ to be any element such that $(\beta,\hat{\beta})=1$ and
	\[
	\norm{\hat{\beta}}_1=\inf\set{\norm{\ga}_1\,:\, \ga\in H_p(M,\mathbb R)\text{ with } (\beta,\ga)=1}.
	\]
	Similarly for $\alpha\in H_p(M,\mathbb R)$ a pairing dual $\hat{\alpha}$ is defined by $(\hat{\alpha},\alpha)=1$ and
	\[
	\norm{\hat{\alpha}}_\infty=\inf \set{\norm{\beta}_\infty\,:\, \beta\in H^p(M,\R)\text{ with } (\beta,\alpha)=1}.
	\]
	By the Hahn-Banach theorem and properties of the norms (see \cite[Proposition F.2.2]{BP92}), pairing duals exist (though not necessarily unique). Moreover, $\norm{\beta}_\infty\cdot \norm{\hat{\beta}}_1=1=\norm{\hat{\alpha}}_\infty\cdot \norm{\alpha}_1$ for all $\alpha\in H_p(M,\R)$ and $\beta\in H^p(M,\R)$. This duality principle is essential to the proof of the above theorem. (See Section \ref{sec:duality-principle}.)
\end{remark}

On the other hand, having an upper bound like Theorem \ref{thm:B-D} will provide positivity of the Gromov norm on corresponding non-trivial classes. In nonpositively curved manifolds, such results are only known in certain degrees and in certain cases such as the negatively curved manifolds, higher rank locally symmetric ones, and rank one manifolds with $\op{Ric}_k<0$. We show in all these cases: 

\begin{theorem}\label{thm:lower_bound_inj_small} Let $n\geq 3$. There is a constant $C(n)>0$, with the following property. Let $M$ be a closed oriented nonpositively curved manifold of dimension $n$, and for $0\leq p\leq n-2$, let $\beta\in H^p(M,\R)$ be any cohomology class. Denote by $\beta^*\in H_{n-p}(M,\R)$ the \Poincare dual of $\beta$. 
	\begin{enumerate}
		\item If $M$ has maximum sectional curvature $-a^2$ with $a>0$, then 
		\[
		\norm{\beta}_H\leq C(n)b_p^{\frac{n-1}{2}}a^{p-n}\max\set{1,\frac{1}{\left(a\op{Inj}(M)\right)^{1/2}}}\norm{\beta^*}_1,
		\]
		where $b_p>0$ is given by $\displaystyle{-b_p^2=\begin{cases} 
			(n-1)^{-1}(\textrm{lower bound of Ricci curvature}) & p\leq 1 \\
			\textrm{lower bound of the curvature operator}& p\geq 2 
			\end{cases}.}$
		\item If $M$ is locally symmetric of real rank $r\geq 2$ normalized to have sectional curvatures $-1\leq K\leq 0$, whose universal cover has no direct factor of $\mathbb R$, $\mathbb H^2$, $\op{SL}(3,\R)/\op{SO}(3)$, $\op{Sp}(2,\R)/\op{U}(2)$, $G_2^2/\op{SO}(4)$ or $\op{SL}(4,\R)/\op{SO}(4)$ and $p\leq n-2-\op{srk}(\til M)$, then
		\[
		\norm{\beta}_H\leq C(n)\max\set{1,\frac{1}{\left(\op{Inj}(M)\right)^{r/2}}}\norm{\beta^*}_1
		\]
		\item If $M$ is geometric rank one satisfying $\op{Ric}_{k+1}<0$ for some $k\leq\floor{\frac{n}{4}}$, then there exists a constant $C(\til{M})$ that depends only on $\til M$ such that when $p\leq n-4k$,
		\[
		\norm{\beta}_H\leq C(\til{M})\max\set{1,\frac{1}{\op{Inj}(M)^{k/2}}}\norm{\beta^*}_1.
		\]
	\end{enumerate}

\end{theorem}

In the statement of the above theorem, the definition of the \emph{splitting rank}, $\op{srk}$, and \emph{$k$-Ricci curvature}, $\op{Ric}_k$, are as follows.

\begin{defn}
	Let $X$ be any symmetric space without compact factors, the splitting rank of $X$ is
	$$\op{srk}(X)=\max\{\;\dim(Y\times \mathbb R)\;|\;\textrm{where}\;Y\times \mathbb R\subset X\; \textrm{is totally geodesic}\;\}$$
\end{defn}
This number is explicitly computed in \cite[Table 1]{Wa16} for all irreducible symmetric spaces. For example, if $X=\op{SL}(m,\mathbb R)/\op{SO}(m)$, then $\op{srk}(X)=\dim(X)-\op{rank}(X)$, and if  $X=\op{SL}(m,\mathbb C)/\op{SO}(m)$, then $\op{srk}(X)=\dim(X)-2\op{rank}(X)$.
\begin{defn}\label{def:l-ricci}
	For $u,v\in T_xM$, the $k$-Ricci tensor is defined to be
	\[
	\op{Ric}_{k}(u,v)=\sup_{\stackrel{V\subset T_xM}{\dim V=k}} \op{Tr}
	R(u,\cdot,v,\cdot)|_V,
	\]
	where $R(u,\cdot,v,\cdot)|_V$ is the restriction of the curvature tensor to
	$V\times V$, and thus the trace is with respect to any orthonormal basis of $V$.
	We set $\op{Ric}_{k}=\sup_{v\in T^1M}\op{Ric}_{k}(v,v)$.
\end{defn}

\begin{remark}\label{rem:bp_control}
	The range of $p$ in Theorem \ref{thm:lower_bound_inj_small}$(a)$ is sharp as the Gromov norm of any $1$-class is zero. It is almost sharp in case $(b)$ due to a recent construction of the second author \cite{Wang:20}, which shows that there is a non-trivial homology class with zero Gromov norm at degree $\op{srk}(\til M)$.
	
	We also note that for a Riemannian manifold with sectional curvatures between $-b^2\leq K\leq 0$, then by Proposition 3.8 of \cite{BourguignonKarcher} we have $b_p\leq (3+\frac{2n}{3})b$. So up to constants depending only on dimension we may replace $b_p$ by $b$ everywhere.
\end{remark}

In the case when the injectivity radius is large we obtain a better bound, in most cases exponentially decaying in the injectivity radius, but for a restricted range of $p$.

\begin{theorem}\label{thm:lower_bound_inj_large} 
	Let $n\geq 3$. There is a constant $C(n)>0$, with the following property. Let $M$ be a closed oriented nonpositively curved manifold of dimension $n$, and for $0\leq p\leq n-2$, let $\beta\in H^p(M,\R)$ be any cohomology class. Denote by $\beta^*\in H_{n-p}(M,\R)$ the \Poincare dual of $\beta$.
	
	\begin{enumerate}
		\item If $M$ has sectional curvatures in $[-b^2,-1]$ and $p<\frac{(n-1)}{2b}$, and $\op{Inj}(M)> 1+\frac{\ln (2)}{(n-1)-2 p b}$, then
		\[
		\norm{\beta}_H\leq \frac{C(n)\sqrt{b}}{\sqrt{n-1-2pb}} e^{-(\frac{(n-1)}2- p b) \op{Inj}(M)}\norm{\beta^*}_1.
		\]
		Moreover, if $b\neq 1$ and $p$ is a non-negative integer such that $p=\frac{(n-1)}{2b}$ and $\op{Inj}(M)>2$, then
		\[
		\norm{\beta}_H \leq C(n)(\op{Inj}(M))^{-\frac12}\norm{\beta^*}_1.
		\]

		\item If $M$ is locally symmetric of real rank $r\geq 2$ whose universal cover has no direct factor of $\mathbb R$, $\mathbb H^2$, $\op{SL}(3,\R)/\op{SO}(3)$, $\op{Sp}(2,\R)/\op{U}(2)$, $G_2^2/\op{SO}(4)$ or $\op{SL}(4,\R)/\op{SO}(4)$, normalized to have $-1\leq K\leq 0$ with Ricci curvature bound $\op{Ric}\leq -\delta^2g$
		and $p<\delta/2$, then 
		\[
		\norm{\beta}_H \leq \frac{C(n)}{\sqrt{\delta-2p}}e^{-\left(\frac{\delta}{2}-p\right)\op{Inj}(M)}\norm{\beta^*}_1.
		\]

		\item If $M$ is geometric rank one satisfying $-b^2\leq K$, $\op{Ric}<-\delta^2 b^2$, $\op{Ric}_{k+1}<0$ for some $k\leq\floor{\frac{n}{4}}$, $0<p\leq \min\set{n-4k,\frac{\delta}{2}}$ and $\op{Inj}(M)>1+\frac{p\log(2)}{\delta-2p}$, then there exists a constant $C(\til{M})$ that depends only on $\til M$ such that
		\[
		\norm{\beta}_H \leq C(\til{M})e^{-\left(\frac{\delta}{2}-p\right)b\op{Inj}(M)}\norm{\beta^*}_1.
		\]
	\end{enumerate}

\end{theorem}

\begin{remarks} 
	In the case when $n=3$, $p=1$ and $M$ is hyperbolic, part (a) of Theorems \ref{thm:lower_bound_inj_small} and \ref{thm:lower_bound_inj_large} recover the estimate in Theorem \ref{thm:B-D} up to uniform multiplicative constants. However, their numerical constants are sharper on both sides.
	
	The irreducible symmetric spaces arising in part (b) of Theorem \ref{thm:lower_bound_inj_large} are Einstein with exiplicit Ricci constant $-\delta^2=-\op{tr}{(\op{ad}_v)^2}$ for any self-adjoint $v$ in the Lie-algebra $\mf{g}$ with unit norm, with respect to the curvature normalized metric. (See e.g. 2.14 of \cite{Eberlein:96} for details.)
	
	In the same setting, as explained in Section 5 of \cite{Brock-Dunfield}, an equivalent form of Theorem 2 of \cite{Kronheimer-Mrowka} states that for any closed oriented irreducible 3-manifold $M$ and any $\beta\in H^1(M;\R)$,
	\[
	2\pi \norm{\beta^*}_1=\inf_g \ \norm{s(g)}_{H_g}\norm{\beta}_{H_g}
	\]
	where the $\inf$ is over all Riemannian metrics $g$ on $M$ with scalar curvature $s(g)$ and harmonic norm $\norm{\cdot}_{H_g}$. Theorems \ref{thm:upper-bound}, \ref{thm:lower_bound_inj_small} and \ref{thm:lower_bound_inj_large} thus provide bounds of how close to and how far from the above infimum general negatively curved metrics on $M$ can become. 
	
	In \cite{Gromov82}, Gromov showed that if $\op{Ric}\geq -(n-1)$ then for all $\alpha\in H_p(M,\R)$ and $\beta\in H^p(M,\R)$ we have
	$$\norm{\alpha}_1\leq p!(n-1)^p\mass(\alpha),\quad\quad \comass \beta\leq p!(n-1)^p\norm{\beta}_\infty$$
	and if the sectional curvatures of $M$ satisfy $K\leq -a^2$ with $a>0$ then
	$$\mass(\alpha)\leq \frac{\pi a^{-p}}{(p-1)!}\norm{\alpha}_1,\quad\quad \norm{\beta}_\infty\leq \frac{\pi a^{-p}}{(p-1)!}\comass (\beta).$$
It follows immediately from the above remarks that we obtain equivalences for the other norms in this context. Specifically, we may replace $\norm{\beta^*}_1$ by the quantities $\norm{\widehat{\beta^*}}_\infty^{-1}$, $\comass(\widehat{\beta^*})^{-1}$ or $\mass(\beta^*)$ in Theorems \ref{thm:upper-bound}, \ref{thm:lower_bound_inj_small} and \ref{thm:lower_bound_inj_large}. However, the constants change by a multiplicative constant in $n$, $b_p$ and $a$. (See the definition of comass in Section \ref{sec:comparing-norms})
\end{remarks}

When passing to a finite cover, neither volume nor injectivity radius of a manifold decreases. Thus our theorems immediately generalize the result of Bergeron, \c{S}eg\"{u}n and Venkatesh.

\begin{corollary}
	For any closed oriented manifold $M_0$ of dimension $n\geq 3$ and integer $p$ satisfying the hypotheses of Theorems \ref{thm:upper-bound} and \ref{thm:lower_bound_inj_small}, there exists constants $C_1$ and $C_2$ depending only on $M_0$ such that for any finite cover $M$ of $M_0$ and all $\beta\in H^p(M,\R)$ with \Poincare dual $\beta^*\in H_{n-p}(M,\R)$, we have
	\[
	\frac{C_1}{\sqrt{\vol(M)}}\norm{\beta^*}_1\leq \norm{\beta}_H\leq C_2\norm{\beta^*}_1.
	\]
\end{corollary}

We structure our paper as follows. In Section \ref{sec:comparing-norms}, we establish an upper bound on the comass of the harmonic form by its harmonic norm in terms of the injectivity radius. This uses the Margulis Lemma together with a Moser typer inequality for a certain elliptic equation. In Section \ref{sec:straightening}, we use the straightening method to relate the harmonic norm and the comass with the Gromov norm of the \Poincare dual. As a result, we prove our Theorems \ref{thm:lower_bound_inj_small} and \ref{thm:lower_bound_inj_large}. The latter theorem also employs the more refined elliptic estimates from \cite{DiCerbo:17}. In Section \ref{sec:duality-principle}, we use the duality principle to prove Theorem \ref{thm:upper-bound}. And lastly in Section \ref{sec:Sobolev}, we give an alternative approach to Section \ref{sec:comparing-norms}, relating the comass and the harmonic norm with other geometric quantities such as the Sobolev constants, isoperimetric type constants and others.

\subsection*{Dictionary of notations} Throughout this paper, all manifolds are oriented. We denote by $\Delta=\delta d+d \delta$ the Hodge Laplacian. We present here a short summary of our notations. Apart from the self-evident ones, and those already defined, these definitions will mostly be explained in Section \ref{sec:comparing-norms} and \ref{sec:Sobolev}.\\
\vskip 0.15 cm
\begin{center}
\begin{tabular}{|c|c||c|c|}
	\hline 
	$\norm{\cdot}_H$ & harmonic norm on $H^p(M,\R)$& $C_0$ & $L^1$-Sobolev constant\\ 
	\hline 
	$\norm{\cdot}_\infty$ & $\ell^\infty$-norm on $H^p(M,\R)$  & $C_1$ & isoperimetric type constant \\ 
	\hline 
	$\norm{\cdot}_1$ & Gromov norm on $H_p(M,\R)$& $\la_1$ & smallest positive eigenvalue for Laplacian  \\ 
	\hline 
	$\comass$ & supremum of $\abs{\cdot}_\infty$ & $h$ & Cheeger constant \\ 
	\hline 
	$\abs{\cdot}_2$ & pointwise $\ell^2$-norm & $C_s$ & $L^2$-Sobolev constant \\ 
	\hline 
	$\abs{\cdot}_\infty$ & pointwise $\ell^\infty$-norm & $d_M$ & diameter of $M$\\ 
	\hline 
	$\beta^*$& \Poincare dual of $\beta$ & $\op{Inj}$ & injectivity radius \\ 
	\hline 
	$\alpha$, $\ga$ & homology classes & $\op{Ric}_k$ & $k$-Ricci \\
		\hline 
	 $\beta$, $\varphi$& cohomology classes  & $b_p$, $K_p$ &  curvature operator (lower) bounds
    \\ 
	\hline 
	$\eta$, $\phi$& differential forms  & $a$, $b$ &  sectional curvature bounds\\ 
	\hline 
	$\omega$& an harmonic form  & $\op{srk}$ & splitting rank \\ 
	\hline 
\end{tabular}
\end{center}

\subsection*{Acknowledgments} We would like to thank Zhichao Wang for pointing out that the pointwise $\ell^2$-norm might not be smooth at the points where it is zero. We especially thank one of the anonymous referees for noting a mistake in the first version in the case of large injectivity radius, and for suggesting the idea of using the sharp estimates of Di~Cerbo and Stern. The second author also thanks the Department of Mathematics at Indiana University and Max Plank Institute for Mathematics for their hospitality while this work was completed.

\vskip 0.15 cm

\section{Comparing harmonic norm with comass}\label{sec:comparing-norms}

\subsection{{Pointwise $\ell^2$-norm and $\ell^\infty$-norm}}

For manifolds, the singular cohomology is isomorphic to the DeRham cohomology.
Denote $\Omega^p(M)$ the space of differential $p$-forms on $M$. For any point
$x\in M$, choose linearly independent local sections $\set{e_1,\dots,e_n}$ of
$TM$ over a neighborhood of $x$ that form an orthonormal basis of $T_xM$, and
let $\set{e_1^*,...,e_n^*}$ be the canonical dual sections on $T^*M$ over the
same neighborhood. For any $\phi\in \Omega^p(M)$, $\phi$ can be locally
expressed as
$$\phi=\sum_{i_1<i_2<...<i_p}a_{i_1 i_2...i_p}\;e_{i_1}^*\wedge e_{i_2}^*\wedge...\wedge e_{i_p}^*.$$  
We define the \emph{pointwise $\ell^2$-norm} at $x$ by
$$\abs{\phi}_2(x)=\left(\sum_{i_1<i_2<...<i_p}a_{i_1 i_2...i_p}^2(x)\right)^{1/2}.$$
Note that the definition is independent of the choice of local sections, provided they are orthonormal at $x$, and it is easy to check that
$$\abs{\phi}_2^2(x)=\frac{\phi\wedge*\phi}{\,d\op{vol}}(x).$$
Thus if $\omega$ is a harmonic form representing $\alpha$, then
$$\norm{\alpha}_H=\left(\int_M \abs{\omega}_2^2\,d\op{vol}\right)^{1/2}=L^2 \textrm{-norm of }\abs{\omega}_2.$$
Similarly, we define the \emph{pointwise $\ell^\infty$-norm}
$$\abs{\phi}_\infty(x)=\sup\set{\phi_x(e_1,...,e_p):\set{e_1,...,e_p} \textrm{ is an orthonormal frame of } T_xM},$$
and for $\phi\in \Omega^p(M)$, we define 
$$\comass{\phi}=\sup_{x\in M}\left(\abs{\phi}_\infty(x)\right).$$ 
For a class $\alpha\in H^p(M,\R)$, define
\[
\comass(\alpha)=\inf\set{\comass(\phi):\phi\in \Omega^p(M), [\phi]=\alpha}.
\]

It is not difficult to see that the pointwise $\ell^2$,$\ell^\infty$-norm are equivalent up to a constant.
\begin{lemma}\label{lem:l2 vs linfty}
	If $\phi\in \Omega^p(M)$ is a differential $p$-form on an $n$-dimensional closed Riemannian manifold $M$, then
	$$\abs{\phi}_\infty(x)\leq \abs{\phi}_2(x)\leq \binom{n}{p}^{1/2}\abs{\phi}_\infty(x).$$
	Thus,
	$$\comass(\phi)\leq \sup{\abs{\phi}_2}\leq \binom{n}{p}^{1/2}\comass(\phi).$$
\end{lemma}

\begin{proof}
	Suppose at $x$, $\phi$ evaluates to its supremum on the orthonormal $p$-frame $\set{e_1,...,e_p}$, we extend the set to an orthonormal $n$-frame $\set{e_1,...,e_p,...,e_n}$, and write $\phi$ (at the point $x$) as
	$$\phi_x=\sum_{i_1<i_2<...<i_p}a_{i_1 i_2...i_p}\;e_{i_1}^*\wedge e_{i_2}^*\wedge...\wedge e_{i_p}^*.$$
	By assumption, $a_{12...p}=\abs{\phi}_\infty(x)$ and $\abs{a_{i_1 i_2...i_p}}\leq \abs{\phi}_\infty(x)$ for all $i_1<i_2<...<i_p$, so we obtain
	$$\abs{\phi}_\infty(x)=a_{12...p}\leq \left(\sum_{i_1<i_2<...<i_p}a_{i_1 i_2...i_p}^2\right)^{1/2}=\abs{\phi}_2(x)\leq \binom{n}{p}^{1/2}\abs{\phi}_\infty(x).$$
\end{proof}

\subsection{Relating harmonic norm with comass}\label{subsec:K_p}
Given a closed Riemannian manifold $M$, we define $K_p$ for any integer $p\geq 0$ to be
\begin{align}\label{def:Kp}
K_p=\begin{cases} 
0 & p=0 \\
(n-1)^{-1}(\textrm{lower bound of Ricci curvature}) & p=1 \\
\textrm{lower bound of the curvature operator}& p\geq 2 
\end{cases}
\end{align}
where the curvature operator is viewed as the linear extension of the curvature tensor to all of $\La^2(T_xM)$ as opposed to the subset $\cup_{v,w\in T_xM} v\wedge w$. We recall from Theorem \ref{thm:lower_bound_inj_small} the definition of $b_p$. In particular if $M$ is nonpositively curved, then $-b_p^2\leq K_p\leq 0$.

We will use the following lemma.

\begin{lemma}[cf. Lemma 8 of \cite{Li}]\label{lem:Li}
Let $M$ be a closed Riemannian manifold of dimension $n$, and $\omega$ be a harmonic $p$-form. If the function $f$ is the square pointwise $\ell^2$-norm $\abs{\omega}_2^2(x)$, then
$$\Delta f\leq \la f$$
where $\la=-2p(n-p)K_p\geq 0$.
\end{lemma}

\begin{proof}
	By Bochner's formula,
\begin{equation}\label{eq:Bochner}
\frac{1}{2}\Delta(\abs{\omega}_2^2)=\inner{\Delta \omega, \omega}-\abs{\nabla\omega}^2-F(\omega)
\end{equation}
where $\Delta=d\delta+\delta d$ is the Hodge Laplacian and $F(\omega)\geq p(n-p)K_p\abs{\omega}_2^2$ by \cite[p. 264]{Gallot-Meyer}. Since $\omega$ is harmonic, we obtain from equation (\ref{eq:Bochner}) that
	\begin{equation}\label{ineq:Bochner}
	\Delta f\leq 2(-\abs{\nabla\omega}^2-p(n-p)K_p f)\leq \la f
	\end{equation}
	where $\la =-2p(n-p)K_p$.
\end{proof}
	
\begin{remark} Note that the pointwise $\ell^2$-norm $\abs{\omega}_2(x)$ is \emph{not} necessarily differentiable at points where it is zero, but its square is always a smooth function.
\end{remark}

The Margulis Lemma states that for any complete simply connected $n$-dimensional Riemannian manifold $X$, with sectional curvatures $-b^2\leq K_X \leq 0$, there are constants $N\in \N$ and $\epsilon>0$, depending only on $b$ and $n$, such that for any discrete group of isometries $\Ga<\Isom(X)$, and any $x\in X$, the group $\Ga_\epsilon$ generated by
\[
\set{\ga\in\Ga \,:\, d(x,\ga x)<\eps}
\]
has a finite index nilpotent subgroup of index at most $N$. The largest value of $\eps$ for a given $b$ and $n$ is called the {\em Margulis constant} $\mu=\mu(b,n)$. Moreover, if $M=X/\Ga$ is a compact manifold then by \cite{Gromoll-Wolf}, $\Ga_\mu$ has a normal crystallographic subgroup $\Sigma_\mu$ of index at most $N$ which preserves and acts cocompactly on a totally geodesic flat submanifold $A\cong \R^k\subset X$. If $X$ is negatively curved then $\Ga_\mu$ is infinite cyclic and $A$ is a geodesic.

Define the {\em thin/thick decomposition} of $M=M_{thin}\cup M_{thick}$ to be that given by the disjoint subsets,
\[
M_{thin}=\set{x\in M\,:\, \Inj(x)< \mu}\quad\text{and}\quad M_{thick}=\set{x\in M\,:\, \Inj(x)\geq \mu}.
\]

If $M$ is compact and negatively curved then $M_{thin}$ is diffeomorphic to a finite disjoint union of copies of the product of an $n-1$-disk with a closed geodesic corresponding to $A$ (possibly together with some twisted product components if $M$ is nonorientable) and $N(n,b)=1$ for $b>0$. Even for non-flat closed nonpositively curved $M$, we may have $M_{thin}=M$, as is the case for some graph manifolds.

The following lemma indicates a useful scale to apply to subsequent results.
\begin{lemma}\label{lem:Margulis}
Suppose $X$ is an $n$-dimensional Hadamard manifold with $n>1$ and Margulis constant $\mu$ and $\Ga<\Isom(X)$ is a cocompact discrete group of isometries. Set $M=X/\Gamma$ and let $\eps=\min\set{\Inj(M),\frac{\mu}{2}}$, and $B=B(x,\frac{\mu}{2})\subset X$ be a metric ball of radius $\frac{\mu}{2}$ centered at some point $x\in X$. If $D=\set{y\in X\,:\, d(x,y)\leq d(x,\ga y) \text{ for all }\ga\in\Ga}$ is the Dirichlet fundamental domain of $\Ga$, centered at $x$, then in each case we have
\begin{enumerate}
	\item If $X$ is negatively curved then $\#\set{\ga\in \Ga\,:\,\ga D\cap B\neq\emptyset}\leq\frac{\mu}{\eps}$
	\item If $X$ has $\Ric_k <0$, then $\#\set{\ga\in \Ga\,:\,\ga D\cap B\neq\emptyset}\leq C\left(\frac{\mu}{\eps}\right)^{k-1}$
	\item If $X$ is a rank $k>1$ symmetric space then $\#\set{\ga\in \Ga\,:\,\ga D\cap B\neq\emptyset}\leq C\left(\frac{\mu}{\eps}\right)^{k}$
	\item If $X$ is any Hadamard space, then $\#\set{\ga\in \Ga\,:\,\ga D\cap B\neq\emptyset}\leq C\left(\frac{\mu}{\eps}\right)^{n}$
\end{enumerate}
where in each case $C$ is a constant only depending on $n$.
\end{lemma}

\begin{proof}
We may assume $\eps=\Inj(M)<\frac{\mu}{2}$, otherwise $B\subset D$ and the left hand side in each case is 1 while the right hand side in each case is at least 2. Similarly we may assume $\Inj(x)\leq \frac{\mu}{2}$ otherwise again $B\subset D$. Hence $B\subset \til M_{thin}$ where $\til M_{thin}\subset X$ is preimage of $M_{thin}$ under the covering map.

The elements $\ga\in\Ga$ which translate $x\in B$ a distance less than the diameter $\mu$ of $B$ lie in $\Ga_\mu$. Hence we can estimate the number of translates of $D$ that intersect $B$ as follows:
\begin{equation*}
\#\{\gamma \in \Gamma|\gamma D\cap B\neq \emptyset\}=\#\{\gamma \in \Gamma_\mu|\gamma D_{\Gamma_\mu}\cap B\neq \emptyset\}\leq N\times \#\{\sigma \in \Sigma_\mu|\sigma D_{\Sigma_\mu}\cap B\neq \emptyset\},
\end{equation*}
where $D_{\Gamma_\mu}$ and $D_{\Sigma_\mu}$ denote the Dirichlet fundamental domains centered at $x$ associated with the groups $\Gamma_\mu$ and $\Sigma_\mu$ respectively. (Recall $\Sigma_\mu<\Gamma_\mu$ of index at most $N$.)
In particular, it can at most be $N$ times the number of translates of a $\Sigma_\mu$-fundamental domain that intersect $B$. 

Moreover, recall that for any flat $A\subset X$ on which $\Sigma_\mu$ acts cocompactly and for any $y\in A$ and $x\in X$, $d(y,\gamma y)\leq d(x,\gamma x)$ holds for all $\gamma\in \Sigma_\mu$ by nonpositive curvature and since $\Sigma_\mu$ stabilizes the flat $A$. Hence the number of points in the orbit of $x\in B$ lying in $B$ is no more than the number of points in the orbit of $y\in A$ in $B(y,\frac{\mu}{2})\cap A$. Since the translation length of each $\ga\in \Ga$ is at least $2\eps$, there can be at most $C(n)\left(\frac{\mu}{\eps}\right)^k$ $\Sigma_\mu$-orbit points in $B$, where $k$ is the rank of $\Sigma_\mu$ which is the dimension of $A$. Here $C(n)$ a-priori depends on the crystallographic group $\Sigma_\mu$, but this lattice density constant (for any shortest translation length $\eps$) is universally bounded from above for each $n$. 

In each of the cases of the lemma we observe that the dimension of $A$ is at most $1,k-1,k,$ and $n$ respectively. The constant $C$ in the Lemma is $C(n)N$ where $N$ is the index of $\Sigma_\mu$. Moreover, when $K_X<0$, both $N=1$ and  $C(n)=1$ since $\Gamma_\mu\cong \Z$.
\end{proof}

Lastly, we will need the following Moser type inequality which represents one of the two directions needed for the Harnack inequality for solutions to general parabolic operators.

\begin{theorem}[Theorem 5.1 of \cite{Saloff-Coste}]\label{thm:Saloff-Coste}
Let $M$ be a complete Riemannian manifold of dimension $n$. If for any  $x\in M$ and $r>0$, the Ricci curvatures restricted to the ball $B(x,2r)$ satisfy $\Ric\geq -(n-1)K$, then for any $u\in C^\infty(B(x,r))$ satisfying $\Delta u \leq \lambda u$ for some $\lambda\geq 0$, we have
\[
\sup_{B(x,r/2)} u\leq C_1 (1+\la r^2)^{1+\frac{n}2} \frac{e^{C_2\sqrt{(n-1)K} r}}{ \vol(B(x,r))}\norm{u\rest{B(x,r)}}_{1},
\]
where $C_1$ and $C_2$ are constants only depending on the dimension $n$. 
\end{theorem}
\begin{proof}
	 This is a special case of \cite[Theorem 5.1]{Saloff-Coste} of Saloff-Coste applied to the restricted setting of the elliptic operator $\scr{L}=\Delta-\lambda \op{I}$ for $\la\geq 0$, instead of the full parabolic setting. In the notation of  \cite{Saloff-Coste}, we are specifically setting $m=1$, $\scr{A}=\op{I}$, $\scr{X}=0=\scr{Y}$, $b=-\la$, $\beta=\la$, $Q=(0,r^2)\times B(x,r)$, $p=1$, and $\delta=\frac12$. It follows that $\alpha=\mu=1$. Finally, we have integrated out the time factor $(0,r^2)$ in $\norm{f\rest{Q}}_{2}$.
\end{proof}

Now we are ready to relate the comass of a harmonic form with its harmonic norm. (Recall the definition of $b_p$ from Theorem \ref{thm:lower_bound_inj_small} (a), and note Remark \ref{rem:bp_control}.)

\begin{theorem}\label{thm:comass-upper-bound}
Let $n\geq 2$. There exists a constant $C(n)$ such that if $M$ is any $n$-dimensional closed nonpositively curved Riemannian manifold, and $\omega$ is any harmonic $p$-form then 
\[
\comass(\omega)\leq C(n)\max\set{b_p^{\frac{n}{2}},\frac{b_p^{\frac{n-\ell}{2}}}{\op{Inj}(M)^{\ell/2}}}\norm{\omega}_2
\]
where,
\begin{enumerate}
	\item if $M$ is negatively curved then $\ell= 1$,
	\item if $M$ has $\Ric_k <0$, then $\ell= k-1$,
	\item if $M$ is locally a rank $k>1$ symmetric space then $\ell= k$, and
	\item in all other cases $\ell= n$.
\end{enumerate}
\end{theorem}
\begin{proof}
Set $\bar{f}(x)=\abs{\omega}_2(x)$ and let $f:\til{M}\to \R$ be the lift of $\bar{f}$ to $\til{M}$. Choose $x_0\in \til M$ to be a point where $f(x_0)=\sup_{x\in \til M}f(x)=\comass(\omega)$. We set $u=f^2$, a smooth function on $\til{M}$.

By Lemma \ref{lem:Li} we have that $u$ is a subsolution of the operator $\Delta-\la  \op{I}$ where $\Delta$ is the Laplacian of $\til M$ and $\la=-2p(n-p)K_p\geq 0$. Since $\Ric_M\geq -(n-1)b_p^2$, by Theorem \ref{thm:Saloff-Coste} applied to $u=f^2$, and taking the square root on both sides, we have for any $r>0$,
\[
\sup_{B(x_0,r/2)} f \leq C_1(1+\la r^2)^{\frac12+\frac{n}4} \frac{e^{C_2\sqrt{(n-1)} b_p r}}{\vol(B(x_0,r))^{\frac12}}\norm{f\rest{B(x_0,r)}}_{2},
\]
where $C_1$ and $C_2$ only depend on $n$. 

We note that the ratio of the right hand side with $\norm{\omega}_2$ grows
exponentially in $r$ for large $r$, since in that case
$\frac{\norm{f\rest{B(x_0,r)}}_{2}}{\vol(B(x_0,r))^{\frac12}}\approx
\norm{\omega}_2$. On the other hand, at least when $\sup f
\gg\norm{\bar{f}}_{2}$, the estimate is likewise poor for small $r$ since then 
$\frac{\norm{f\rest{B(x_0,r)}}_{2}}{\vol(B(x_0,r))^{\frac12}}\approx f(x_0)$.
However, for $r=\frac{\mu}{2}$, where $\mu$ is the Margulis constant, the term
$C(r,b_p,n)=C_1(1+\la r^2)^{\frac12+\frac{n}4} \frac{e^{C_2\sqrt{(n-1)} b_p r}}{
\vol(B(x_0,r))^{\frac12}}$ depends only on $b_p$ and $n$ since $\mu$ does. Moreover, varying $b_p$ by scaling the metric implies that $2r=\mu=\frac{C_3}{b_p}$ where $C_3$ depends only on $n$ and $p$. Recalling that $\la r^2=-2p(n-p)K_p\frac{C_3}{4 b_p^2}$, we obtain $C(r,b_p,n)=C_4 b_p^{\frac{n}{2}}$ for some $C_4$ depending only on $n$ and $p$.

Lemma \ref{lem:Margulis} states that $B(x,\frac{\mu}{2})$ is
contained in a union of at most $C_5\left(\frac{\mu}{\epsilon}\right)^\ell$, for each case of
$\ell$, fundamental domains $D\subset \til M$ of $M$ where
$\epsilon=\min\set{\Inj(M),\frac{\mu}{2}}$ and $C_5$ depends only on $n$. Hence,
\[
\comass(\omega)\leq C_4 b_p^{\frac{n}{2}}\norm{f\rest{B(x,\frac{\mu}{2})}}_{2}\leq C_4 b_p^{\frac{n}{2}} \sqrt{C_5\left(\frac{\mu}{\epsilon}\right)^\ell}\norm{f\rest{D}}_{2} = C_4 b_p^{\frac{n}{2}} \sqrt{C_5}\left(\frac{\mu}{\epsilon}\right)^{\ell/2}\norm{\omega}_{2}.
\]
The statement of the theorem follows from $\left(\frac{\mu}{\epsilon}\right)^{\ell/2}\leq \max\set{2^{\ell/2},\left(\frac{C_3}{b_p\Inj(M)}\right)^{\ell/2}}$ and that $p\leq n$.
\end{proof}

\section{Producing lower bounds on the Gromov norm by simplex straightening}\label{sec:straightening}

The idea of simplex straightening was first developed by Gromov \cite{Gromov82} and Thurston \cite{Thurston77} in negative curvatures as a tool for obtaining lower bounds on the Gromov norm (also see \cite{IY82}). This method has been extended by many authors in different contexts including higher rank symmetric spaces \cite{LS06, KK15, LW15, Wa16}, certain nonpositively curved manifolds \cite{CW17, CW18} and others \cite{Min}. Below we give a brief overview on some of these results.

We begin with the definition of a straightening. 

\begin{defn}\label{def:straightening} (Compare \cite{LS06}) Let $\til{M}$ be the universal
	cover of a connected $n$-dimensional manifold $M$. We denote by $\Gamma$ the fundamental group
	of $M$, and by $C_{\ast}(\til{M})$ the real singular chain complex of
	$\til{M}$. Equivalently, $C_{k}(\til{M})$ is the free $\R$-module
	generated by $C^0(\Delta^k,\til{M})$, the set of singular $k$-simplices in
	$\til{M}$, where $\Delta^k$ is equipped with some fixed Riemannian metric. We say a
	collection of maps $st_k:C^0(\Delta^k,\til{M})\rightarrow C^0(\Delta^k,\til{M})$
	is a straightening if it satisfies the following conditions:
	\begin{enumerate}
		\item the maps $st_k$ are $\Gamma$-equivariant,
		\item the maps $st_{\ast}$ induce a chain map $st_{\ast}:C_{\ast}(\til{M},\R)\rightarrow C_{\ast}(\til{M},\R)$ that is $\Gamma$-equivariantly chain
		homotopic to the identity,
		\item the image of $st_n$ lies in $C^1(\Delta^n, \til{M})$, that is, the top
		dimensional straightened simplices are $C^1$. (Hence all the straightened $k$-simplices are $C^1$.)
	\end{enumerate}
In addition, if all the straightened $k$-simplices have uniformly bounded volume, we say the straightening is \emph{$k$-bounded}.
\end{defn}

When $M$ is negatively curved, Gromov and Thurston used the ``geodesic straightening'' and showed it is $k$-bounded for $k\geq 2$. More specifically, given any $k$-simplex with ordered vertices $\{v_0, v_1,..., v_k\}$, we connect $v_k$ with $st_{k-1}(\{v_0, v_1,..., v_{k-1}\})$ by geodesics with the obvious parametrization, where $st_{k-1}(\{v_0, v_1,..., v_{k-1}\})$ is the inductively defined straightened $(k-1)$-simplex. This geodesic coning procedure defines a straightening that has explicit volume control on each straightened simplex.

When $M$ is higher rank locally symmetric, Lafont and Schmidt \cite{LS06} defined the ``barycentric straightening.'' Using this procedure and a previous estimate of Connell and Farb \cite{CF1, CF2}, they showed that apart from a few exceptional cases, all top dimensional straightened simplices have uniformly bounded volume. This has been extended to all $k$-simplices when $k\geq \textrm{srk}(X)+2$ in \cite{LW15,Wa16}, where \emph{the splitting rank}, denoted by $ \textrm{srk}(X)$, is defined to be the maximal dimension among all totally geodesic submanifolds in $X$ that split off a direct $\R$-factor. (See explicitly \cite[Table 1]{Wa16}.) The idea of barycentric straightening is based on the barycenter method originally developed by Besson, Courtois, and Gallot \cite{BCG}. 

When $M$ is a nonpositively curved rank one manifold, the barycentric straightening is also well defined. If the manifold satisfies additional curvature conditions, then the straightening is also $k$-bounded when $k$ is large enough. We summarize the above discussion into the following proposition.

\begin{theorem}\label{thm:straight} Let $M$ be a closed nonpositively curved manifold of dimension $n$.
	\begin{enumerate}
		\item \cite{Gromov82}
		If M has sectional curvatures $K \leq -a^2$ with $a>0$, then the geodesic straightening is $k$-bounded for $k\geq 2$. Moreover, the volume of the straightened simplices satisfy
		$$\vol(st_{k}(f))\leq \frac{\pi a^{-k}}{(k-1)!}\quad\quad \forall f\in C^0(\Delta^k,\til{M}), \quad  k\geq 2.$$  
		\item \cite{LW15, Wa16} If $M$ is higher rank locally symmetric whose universal cover $\til M$ has no direct factor of $\mathbb R$, $\HH^2$, $\op{SL}(3,\R)/\op{SO}(3)$, $\op{Sp}(2,\R)/\op{U}(2)$, $G_2^2/\op{SO}(4)$ and $\op{SL}(4,\R)/\op{SO}(4)$, then the barycentric straightening is $k$-bounded for $k\geq \op{srk}(\til M)+2$.
		\item \cite{CW17} If $M$ is geometric rank one with $\op{Ric}_{l+1}<0$ for some $l\leq\floor{\frac{n}{4}}$, then the barycentric straightening is $k$-bounded for $k\geq 4l$.
	\end{enumerate}
\end{theorem}

Now we illustrate how the straightening gives a lower bound on the Gromov norm.

\begin{theorem}\label{thm:lower-bound}
Let $M$ be a closed nonpositively curved manifold of dimension $n$, and $\omega$ be the harmonic representative of $\beta\in H^p(M,\R)$ where $p\leq n-2$. Denote $\beta^*\in H_{n-p}(M,\R)$ the \Poincare dual of $\beta$.
\begin{enumerate}
\item If $M$ is negatively curved and has sectional curvatures $K \leq -a^2$ with $a>0$, then
$$\norm{\omega}_2^2\leq \frac{\pi a^{p-n}}{(n-p-1)!}\comass(\omega)\norm{\beta^*}_1.$$
\item If $M$ is higher rank locally symmetric whose universal cover has no direct factor of $\mathbb R$, $\HH^2$, $\op{SL}(3,\R)/\op{SO}(3)$, $\op{Sp}(2,\R)/\op{U}(2)$, $G_2^2/\op{SO}(4)$ and $\op{SL}(4,\R)/\op{SO}(4)$, and $p\leq n-2-\op{srk}(\til M)$, then there exists a constant $C(\til{M})$ that depends only on $\til M$ such that
$$\norm{\omega}_2^2\leq C(\til{M})\comass(\omega)\norm{\beta^*}_1.$$
\item If $M$ is geometric rank one satisfying $\op{Ric}_{k+1}<0$ for some $k\leq\floor{\frac{n}{4}}$, and $p\leq n-4k$, then there exists a constant $C(\til{M})$ that depends only on $\til M$ such that
$$\norm{\omega}_2^2\leq C(\til{M})\comass(\omega)\norm{\beta^*}_1.$$
\end{enumerate}
\end{theorem}

\begin{proof}
(a) By definition, the square of the harmonic norm is given by
$$\norm{\omega}_2^2=\int_M \omega\wedge *\omega=\int_{\beta^*}*\omega,$$
where $\beta^*\in H_{n-p}(M,\R)$ denotes the \Poincare dual of $\beta$. Suppose $\sum_{i}a_i\sigma_i$ is a chain that represents $\beta^*$, since the straightening is $\Gamma$-equivariant it descents to $M$, and by using (b) of Definition \ref{def:straightening} we can replace each $\sigma_i$ by geodesically straightened simplices $st_{n-p}(\sigma_i)$ so that
$$\int_{\beta^*}*\omega=\int_{[\sum_{i}a_i\sigma_i]}*\omega=\int_{[\sum_{i}a_i st(\sigma_i)]}*\omega$$
By lifting to the univeral cover $\til M$, we have
$$\int_{[\sum_{i}a_i st(\sigma_i)]}*\omega\leq \sum_{i}\abs{a_i}\abs{\int_{{st(\sigma_i)}}*\omega}\leq \sum_i\abs{a_i}\comass(*\omega)\vol({st(\til{\sigma_i})})$$
By (a) of Theorem \ref{thm:straight} together with the fact $\comass(*\omega)=\comass(\omega)$, we obtain
$$\sum_i\abs{a_i}\comass(*\omega)\vol({st(\til {\sigma_i})})\leq \frac{\pi a^{p-n}}{(n-p-1)!}\comass(\omega)\left(\sum_i\abs{a_i}\right).$$
Finally by taking the infimum among all chains representing $\beta^*$, and combining the above inequalities, we conclude
$$\norm{\omega}_2^2\leq \frac{\pi a^{p-n}}{(n-p-1)!}\comass(\omega)\norm{\beta^*}_1.$$

(b) and (c) follows similarly by simply replacing geodesic straightening with barycentric straightening, and (a) of Theorem \ref{thm:straight} with (b) and (c).

\end{proof}

Combining Theorem \ref{thm:comass-upper-bound} with Theorem \ref{thm:lower-bound} directly yields Theorem \ref{thm:lower_bound_inj_small}.

Next we consider the small injectivity radius case. For this, we employ the sharp estimates of Di~Cerbo and Stern (\cite{DiCerbo:17}). We summarize their results into the form we will use below.

\begin{theorem}[Theorems 39,87 and 96 of \cite{DiCerbo:17}]\label{thm:Dicerbo}
Let $n\geq 3$. There is a constant $C(n)>0$, with the following property. Let $M$ be a closed oriented nonpositively curved manifold of dimension $n$, and for $0\leq p\leq n-2$, let $\omega$ be any harmonic differential $p$-form on $M$.

	\begin{enumerate}
		\item If $M$ has curvatures in $[-b^2,-1]$ with $b\geq 1$, $p< \frac{(n-1)}{2b}$, and $\op{Inj}(M)> 1+\frac{\ln (2)}{(n-1)-2 p b}$, then for any $x\in M$,
		\[
		\int_{B_{1}(x)}|\omega|^{2}_2 \,d\op{vol} \leq \frac{C(n)b}{n-1-2pb} e^{-((n-1)-2 p b) \op{Inj}(M)}\|\omega\|_{L^{2}(M, g)}^{2}.
		\]
		Moreover, if $b\neq 1$ and $p$ is a non-negative integer such that $p=\frac{n-1}{2b}$ and $\op{Inj}(M)>2$, then
		\[
		\int_{B_{1}(x)}|\omega|^{2}_2 \,d\op{vol} \leq C(n)e^{2b}(\op{Inj}(M))^{-1}\|\omega\|_{L^{2}(M, g)}^{2}.
		\]
		
		\item If $n>5$, $M$ has curvatures $-1\leq K\leq 0$, $-\op{\Ric}\geq \delta^2 g$, $p<\frac{\delta}{2}$ and $\op{Inj}(M)>1+\frac{p\log(2)}{\delta-2p}$, then for any $x\in M$,
		\[
		\int_{B_{1}(x)}|\omega|^{2}_2 \,d\op{vol} \leq \frac{C(n)}{\delta-2p}e^{-\left(\delta-2p\right)\op{Inj}(M)}\|\omega\|_{L^{2}(M, g)}^{2}.
		\]
	\end{enumerate}
\end{theorem}

\begin{proof}
	The two cases in part (a) correspond to Theorems 87 and 96 of \cite{DiCerbo:17}, respectively, with the choices  $k=p$, $\sigma=1$, $\tau=\op{Inj}(M)$, $\eps=b-1$. (Note that under the hypotheses $e^{-(\tau-\sigma)(n-1-2pb)}\leq \frac12$.) This is a similar statement to Theorem 23 of \cite{DiCerbo:19}. 
	
	For part (b) we employ Theorem 39 of \cite{DiCerbo:17} with $\kappa=0$, $\sigma=1$, $\tau=\op{Inj}(M)$, and see the last line of the proof of Theorem 39 regarding why we can choose their $r_0$ to be $\infty$ in the $\kappa=0$ case. We also renormalize for arbitrary curvature bounds, simplify the expressions with dependent constants by adding in the appropriate multiple of the integral over $B_1(x)$ to both sides of the inequalities.
\end{proof}

Before we prove Theorem \ref{thm:lower_bound_inj_large} we will need the following lemma.

\begin{lemma}\label{lem:Ricci_vs_srk}
For any of the symmetric spaces $\til M$ in case (b) of Theorem \ref{thm:lower_bound_inj_large}, scaled to have sectional curvatures $-1\leq K \leq0$, if the Ricci curvature is bounded by $-\op{Ric}\geq \delta^2g$, then \[\floor{\frac{\delta}{2}}\leq n-2-\op{srk}(\til M).\]
\end{lemma}

\begin{proof}
According to the definition of $\til M$, there exists a unit tangent vector $v\in T_o^1\til M$ at some point $o\in \til M$ and a totally geodesic submanifold $Y\times \mathbb R\subset \til M$ where $v$ is tangent to the $\mathbb R$-direction such that $\dim(Y\times \R)=\op{srk}(\til M)$. We compute the Ricci curvature in direction $v$ with respect to a specific orthonormal frame $\{e_1,...,e_k,e_{k+1},...,e_n\}$ such that $e_1=v$, $e_2,...,e_k \in T_oY$ and $e_{k+1},...,e_n\in (T_oY)^\perp$ where $k=\op{srk}(\til M)$,
\begin{align*}
\Ric(v)=\sum_{i=1}^n \langle R(v,e_i)e_i,v\rangle&=\sum_{i=1}^k \langle R(v,e_i)e_i,v\rangle+\sum_{i=k+1}^n \langle R(v,e_i)e_i,v\rangle\\
&\geq 0+(n-k)\cdot (-1)\\
&=-(n-\op{srk}(\til M))
\end{align*}
where the inequality uses the sectional curvature bound together with the fact that $Y\times \R$ is a totally geodesic Riemannian product. Thus, we have $\delta^2\leq -\op{Ric}(v)\leq n-\op{srk}(\til M)$, and it follows that
 \begin{equation}\label{eq:delta_vs_srk}
 \delta\leq \sqrt{n-\op{srk}(\til M)}.
 \end{equation}
Note that the following inequality holds
\[\frac{\sqrt{n-\op{srk}(\til M)}}2\leq n-2-\op{srk}(\til M),\]
provided $\op{srk}(\til M)\leq n-3$, hence the lemma follows in this case. If $\op{srk}(\til M)=n-2$, then inequality \eqref{eq:delta_vs_srk} implies that $\delta\leq \sqrt 2$, so the inequality also holds. In the remaining cases, $\op{srk}(\til M)=n$ would imply that $\til M$ has an $\R$-factor, and 
$\op{srk}(\til M)=n-1$ would imply $\til M$ has an $\mathbb H^2$-factor, which are excluded by the assumption.
\end{proof}

\begin{proof}[Proof of Theorem \ref{thm:lower_bound_inj_large}]
Set $f(x)=\abs{\omega}_2(x)$ and  set $u=f^2$. Choose $x_0\in \til M$ to be a point where $f(x_0)=\sup_{x\in \til M}f(x)=\comass(\omega)$. 

By Lemma \ref{lem:Li} we have that $u$ is a subsolution of the operator $\Delta-\la  \op{I}$ where $\Delta$ is the Laplacian of $\til M$. Since in all cases $\Ric_M\geq -(n-1)b_p^2$, we may first apply Theorem \ref{thm:Saloff-Coste} to $u=f^2$ with $r=1$ to obtain,
\[
\op{comass}(\omega)^2=\sup_{x\in M} \abs{\omega}^2_2(x)\leq C(b_p,n,p) \int_{B_1(x_0)}\abs{\omega}^2_2(x)\,d\op{vol}. 
\]
For case (a) we rescale the metric to have an upper curvature bound of $-1$ instead of $-a^2$. Under this curvature assumption and the stated Ricci curvature bounds, we may apply Theorem \ref{thm:Dicerbo} in each case to obtain the bound on $\op{comass}(\omega)$ analogous to Theorem \ref{thm:comass-upper-bound} in each case.

For case (b) we apply Lemma \ref{lem:Ricci_vs_srk} to obtain the stated bounds on $p$ also imply $p\leq n-2-\op{srk}(\til{M})$. We also observe that in case (b) all of the symmetric spaces have dimension at least six. Finally, applying Theorem \ref{thm:lower-bound} completes the result.
\end{proof}

\section{Producing upper bounds on the Gromov norm}\label{sec:duality-principle}
Consider the set of all singular $p$-simplices $\Sigma^p=\set{\sigma:\Delta^p\rightarrow M}$, a cochain $c\in C^p(M,\R)$ has a natural $\ell^\infty$-norm,
$$\norm{c}_\infty:=\sup_{\sigma\in \Sigma^p}\abs{c(\sigma)}.$$
Thus for cohomology class $\alpha\in H^p(M,\R)$, the $\ell^\infty$-norm is defined to be 
$$\norm{\alpha}_\infty:=\inf\set{\norm{c}_\infty:{c\in C^p(M,\R),\delta c=0, [c]=\alpha}}.$$
Note that by definition $\norm{\alpha}_\infty$ can take infinite values, and we say $\alpha$ \emph{is bounded} if $\norm{\alpha}_\infty$ is finite, or equivalently, $\alpha$ has a bounded representative.

In the theory of bounded cohomology, the $\ell^\infty$-norm on the cohomology is considered dual to the $\ell^1$-norm on the homology. In particular, surjectivity of the comparison map on top dimension is equivalent to the nonvanishing of the simplicial volume. More generally, we have duality principle
\begin{prop}\label{prop:duality}\cite{Gromov82}\cite[Proposition F.2.2]{BP92} For any $\alpha\in H_k(M,\R)$, the Gromov norm satisfies
	$$\norm{\alpha}_1=\frac{1}{\inf\{\;\norm{\varphi}_\infty\vert\;(\varphi,\alpha)=1, \varphi\in H^k(M,\R)\; \}}$$
\end{prop}

Therefore, in order to produce an upper bound on the Gromov norm, it is equivalent to obtain a lower bound on the $\ell^\infty$-norm of certain cohomology classes. It is shown that whenever the manifold has Ricci curvatures bounded from below, the $\ell^\infty$-norm always bounds the comass.

\begin{lemma}\label{lem:upper-bound-on-comass}\cite[p. 244 Corollary]{Gromov82}
	 Let $M$ be a complete Riemannian manifold of dimension $n$ with $\op{Ric}\geq -(n-1)$, then for any $\varphi\in H^p(M,\R)$, we have
	$$\comass{(\varphi)}\leq p!(n-1)^p\norm{\varphi}_\infty$$
\end{lemma}

\begin{remark}
	In \cite{Gromov82}, Gromov used a different normalization, namely $Ric\geq -\frac{1}{n-1}$, but since $\comass(\varphi)$ scales by $(n-1)^p$ when the metric is scaled down by $n-1$, and $\norm{\varphi}_\infty$ stays the same, the resulting inequality has an additional factor of $(n-1)^p$. In a recent preprint of Campagnolo and the second author \cite{CW22}, the above inequality is further sharpened by an additional factor of $1/p^{p/2}$, hence the upper bound of Theorem \ref{thm:upper-bound} can be improved accordingly by a factor of $1/(n-p)^{(n-p)/2}$.
\end{remark}
Thus combining the above two results, we prove Theorem \ref{thm:upper-bound}.
\begin{theorem}
	Let $M$ be a closed manifold of dimension $n$ with $\op{Ric}\geq -(n-1)$, and $\omega$ be the harmonic representative of $\beta\in H^p(M,\R)$, then the Gromov norm of the \Poincare dual of $\beta$ has an upper bound
	$$\norm{\beta^*}_1\leq (n-p)!(n-1)^{n-p}\norm{\omega}_2\sqrt{\vol(M)}.$$
\end{theorem}
\begin{proof}
	Let $\varphi\in H^{n-p}(M,\R)$ be any class such that $(\varphi,\beta^*)=1$, and $\eta$ be any $(n-p)$-form representing $\varphi$, then we have
	$$1=(\beta^*,\varphi)=\int_{\beta^*}\varphi=\int_M \omega\wedge \eta\leq \comass{(\eta)}\int_{M}\abs{\omega}_\infty(x)\,d\op{vol}(x).$$
	By taking the infimum of comass on all $\eta$ representing $\varphi$, we get
	$$1\leq \comass{(\varphi)}\int_{M}\abs{\omega}_\infty(x)\,d\op{vol}(x).$$
	Apply Lemma \ref{lem:l2 vs linfty} and the Cauchy-Schwarz inequality, to further obtain
	$$1\leq \comass{(\varphi)}\int_{M}\abs{\omega}_2(x)\,d\op{vol}(x)\leq \comass{(\varphi)}\norm{\omega}_2\sqrt{\vol(M)}.$$
	Applying Lemma \ref{lem:upper-bound-on-comass}, we have
	$$1\leq (n-p)!(n-1)^{n-p}\norm{\varphi}_\infty\norm{\omega}_2\sqrt{\vol(M)}.$$
	Finally we take the infimum of $\norm{\varphi}_\infty$ over all $\varphi$ satisfying $(\varphi,\beta^*)=1$, and apply Proposition \ref{prop:duality} to get
	$$\norm{\beta^*}_1\leq (n-p)!(n-1)^{n-p}\norm{\omega}_2\sqrt{\vol(M)}.$$	
	
\end{proof}

\section{Alternative approach}\label{sec:Sobolev} In Section \ref{sec:comparing-norms}, we compared the comass of a harmonic $p$-form with its $L^2$-norm in terms of the injectivity radius of the manifold. We can also relate them by other geometric invariants such as the Sobolev constant. Such relation follows from a general theorem of Li.

\subsection{Li's theorem}

Let $K_p$ be the constant defined in Equation \eqref{def:Kp}.
In our case, $K_p$ is always nonpositive but may be less than the lower Ricci bound. We let $V$ be the volume of $M$ and let $C_S$ denote the Sobolev constant which will be defined in the next subsection. We will also relate it to other geometric quantities such as diameter, Cheeger's isoperimetric constant, and the bottom of the spectrum of the Laplacian.

\begin{theorem}\cite[Theorem 7]{Li}\label{thm:Li}
	There exists a constant $C(n)$ depending only on $n$ such that if $\phi$ is a $\la$-eigenform for the Hodge Laplacian on a closed manifold $M$ of dimension $n\geq 3$ and $\la\neq p(n-p)K_p$ then
	\[
	C(n)\left(\frac{\la-p(n-p)K_p}{C_S}\right)^{n/2}\exp\left[\frac{C(n)C_S}{V^{2/n}(\la-p(n-p)K_p)}\right] \norm{\phi}_2^2 \geq \left(\sup\abs{\phi}_2\right)^2.
	\]
	Moreover, if $\la=p(n-p)K_p$ then $\la=0=K_p$ and $\abs{\phi}_2$ is pointwise constant. In particular, $\norm{\phi}_{2}^2=V\left(\sup\abs{\phi}_2\right)^2$.
\end{theorem}

For our purpose, we only need to apply Li's theorem in the case of harmonic $p$-forms, that is, $\lambda=0$.

\begin{corollary}\label{cor:Li}
	If $M$ is a closed manifold of dimension $n\geq 3$, then there exists a constant $C(n)$ depending only on $n$ such that for any harmonic $p$-form $\omega$ with $1\leq p\leq n-1$,
	\[
	\comass(\omega)\leq \sqrt{Q}\norm{\omega}_2,
	\]
	where $Q$ is given by
	\[
	Q=C(n)\left(\frac{-p(n-p)K_p}{C_S}\right)^{n/2}\exp\left[\frac{C(n)C_S}{-p(n-p)K_p V^{2/n}}\right]\quad\text{ if $K_p\neq 0$ }\quad\text{ or }\quad Q=V\text{ if }\; K_p=0.
	\] 
\end{corollary}

\begin{proof}
	Combine Lemma \ref{lem:l2 vs linfty} with Theorem \ref{thm:Li}.
\end{proof}

\begin{remark}
	In Theorem 7 of \cite{Li} the appendix shows that $C(n)=D(n)^{\frac{\beta}{\beta-1}}$ where $\beta=\frac{n}{n-2}$ and  $D(n)$ is the constant from Lemma 2 of that paper. That constant is $D(n)=2^\alpha$ where $\alpha=0$ if $n=3$ and $\alpha=\frac{n-4}{n-2}$ for $n\geq 4$. Therefore, in both the above Theorem \ref{thm:Li} and Corollary \ref{cor:Li}, $C(n)\leq 1$ if $n=3$ and otherwise $C(n)\leq 2^{\frac{n^2-4n}{2}}$.
\end{remark}

\subsection{Sobolev constant, Cheeger's constant and bottom of the spectrum of the Laplacian}

Let $C_0$ be the classical $L^1$-Sobolev constant for the Sobolev space $W^{1,1}$ of functions on an $n$-dimensional closed manifold $M$ with weak $L^1$ first derivatives, that is,
\[
C_0=\inf_{f\in W^{1,1}}\frac{\norm{\grad f}_1^n}{\inf_{a\in\R}\norm{f-a}_{\frac{n}{n-1}}^n}. 
\]
This is related to a scale invariant isoperimetric type constant $C_1$,
\[
C_1=\inf_{N\subset M}\frac{\left(\vol_{n-1}(N)\right)^n}{\min\set{\vol{M_1},\vol{M_2}}^{n-1}}
\]
where the infimum is over all codimension one closed submanifolds $N\subset M$ whose complement in $M$ consists of two components $M_1$ and $M_2$. The relation is
\begin{equation}\label{eq:C0 with C1}
C_1\leq C_0 \leq 2 C_1.
\end{equation}
Another well known isoperimetric constant is the Cheeger's constant $h$, it is defined in a similar way as
$$h=\inf_{N\subset M}\frac{\vol_{n-1}(N)}{\min\set{\vol{M_1},\vol{M_2}}}.$$
Hence it follows that
\begin{equation}\label{eq:C1 with h}
C_1\leq \frac{h^n\vol(M)}{2}.
\end{equation}
Moreover, we define $C_S$ by
\[
C_S=\left(\frac{2(n-1)}{n-2}\right)^{\frac2n} C_0^{\frac2n}.
\]
Lemma 1 of \cite{Li} shows
\[
C_S\leq \inf_{f\in W^{1,2}} \frac{\norm{\grad f}_2^2}{\norm{f}_{\frac{2n}{n-2}}^2},
\]
where the right hand side is the $L^2$ Sobolev constant. 

If we denote $\la_1$ the smallest positive eigenvalue of the Laplacian on $M$, then it is bounded in terms of the Cheeger's constant $h$.

\begin{theorem}\cite{Cheeger, Buser}\label{thm:Cheeger-Buser}
	$$\la_1\geq \frac{h^2}{4}\quad\quad\quad\textrm{Cheeger's inequality}$$
	If $M^n$ has Ricci lower bound $-(n-1)$, then
	$$\la_1\leq 2(n-1)h+10h^2\quad\quad\quad\textrm{Buser's inequality}$$
\end{theorem}

We attempt to give an upper bound on the constant $Q$ in Corollary \ref{cor:Li}, in terms of these geometric quantities. First we observe that the exponential part of $Q$ is bounded for non-flat nonpositively curved manifolds.

\begin{lemma}\label{lem:exp-is-bounded} If $M$ is a closed nonpositively curved manifold which is not flat, $C(n)$ is a constant that depends only on $n$, and $1\leq p\leq n-1$, then
	$$\exp\left[\frac{C(n)C_S}{-p(n-p)K_p V^{2/n}}\right]\leq C(n,K_p),$$
	where  $C(n,K_p)$ is a constant that only depends on $n$ and $K_p$.
\end{lemma}

\begin{proof}
	Since $C(n)$ and $p$ only depend on $n$, and $K_p<0$ under the hypotheses, it suffices to bound $C_S/V^{2/n}$ from above. 	In what follows, let $\op{const}_n$ represent a mutable constant depending only on $n$. According to the definition $C_S$ and inequality (\ref{eq:C0 with C1}),
	$$C_S/V^{2/n}\leq \op{const}_n\left(\frac{C_0}{V}\right)^{2/n}\leq \op{const}_n\left(\frac{C_1}{V}\right)^{2/n}.$$
	By inequality (\ref{eq:C1 with h}) and Cheeger's inequality (\ref{thm:Cheeger-Buser}),
	$$\left(\frac{C_1}{V}\right)^{2/n}\leq \op{const}_n h^2\leq \op{const}_n\la_1.$$
    By Cheng's comparison theorem \cite[Corollary 2.3]{Cheng75},
    $$\la_1\leq \frac{(n-1)^2K}{4}+\frac{\op{const}_n}{d_M^2}$$
    where $-(n-1)K$ is the Ricci lower bound of $M$ (hence $K\leq -K_p$).
    
    Consider a closed Dirichlet fundamental domain $D\subset \til{M}$ for $\pi_1(M)$ centered at some $x\in \til{M}$ as defined in Lemma \ref{lem:Margulis}. Since $M$ is closed and not flat, $\pi_1(M)$ is not infra-nilpotent by \cite{Gromoll-Wolf}. Hence, in the generating set $S=\set{\ga\in\pi_1(M)\,:\, \ga D\cap D\neq \emptyset}$ there is at least one element $\ga\in S$ with translation length greater than the Margulis constant $\mu$. Therefore the  diameter of $M$ satisfies $d_M\geq \mu$, and $\mu$ only depends on $n$ and $b_p=\sqrt{-K_p}$. Therefore we obtain an upper bound on $\la_1$ depending only on $n$ and $K_p$. Hence the same holds for the upper bound of $C_S/V^{2/n}$, which completes the proof.
\end{proof}

Now contracting all constants in Corollary \ref{cor:Li} that depend only on $n$ and $K_p$, we obtain

\begin{corollary}
		If $M$ is a closed nonpositively curved manifold of dimension $n\geq 3$ which is not flat, then there exists a constant $C(n,K_p)$ depending only on $n$ and $K_p$ such that for any harmonic $p$-form $\omega$ representing $\beta\in H^p(M,\R)$ with $1\leq p\leq n-1$, we have
		\[
		\comass(\omega)\leq \frac{C(n,K_p)}{C_S^{n/4}}\norm{\omega}_2.
		\]
		Thus in addition if $M$ and $p$ satisfy any of the following conditions
		\begin{enumerate}
			\item If $M$ is negatively curved, and $p\leq n-2$, or
			\item If $M$ is rank $r\geq 2$ locally symmetric whose universal cover has no direct factor of $\HH^2$, $\op{SL}(3,\R)/\op{SO}(3)$, $\op{Sp}(2,\R)/\op{U}(2)$, $G_2^2/\op{SO}(4)$ and $\op{SL}(4,\R)/\op{SO}(4)$, and $p\leq n-2-\op{srk}(\til M)$, or
			\item If $M$ is geometric rank one satisfying $\op{Ric}_{k+1}<0$ for some $k\leq\floor{\frac{n}{4}}$, and $p\leq n-4k$,
		\end{enumerate}
		then there exists a constant $C(\til M)$ that depends only on $\til M$ such that
		$$\norm{\omega}_2\leq \frac{C(\til M)}{C_S^{n/4}}\norm{\beta^*}_1.$$
\end{corollary}

\begin{proof}
	Combine Lemma \ref{lem:exp-is-bounded}, Corollary \ref{cor:Li} and Theorem \ref{thm:lower-bound}.
\end{proof}
\begin{remark}
	Note that since $C_S\simeq C_0^{2/n}\simeq C_1^{2/n}$ up to universal constants, the inequality in the above Corollary can be replaced by
	$$\comass(\omega)\leq \frac{C(n,K_p)}{\sqrt C_0}\norm{\omega}_2\quad\quad \textrm{or} \quad\quad \comass(\omega)\leq \frac{C(n,K_p)}{\sqrt C_1}\norm{\omega}_2,$$
	and also
	$$\norm{\omega}_2\leq \frac{C(\til M)}{\sqrt C_0}\norm{\beta^*}_1\quad\quad \textrm{or}\quad\quad \norm{\omega}_2\leq \frac{C(\til M)}{\sqrt C_1}\norm{\beta^*}_1$$
\end{remark}

We end the section by a result of Croke which provides a lower bound on $C_1$ (so equivalently on $C_0$ and $C_S$) in terms of the lower bound of Ricci curvature $(n-1)K$, upper bound of diameter $d_M$, lower bound of $\vol(M)$ and the dimension $n$.

\begin{theorem}\cite[Theorem 13]{Croke} For any closed Riemannian manifold with Ricci curvature bounded below by $(n-1)K$, there is a constant $C(n)$ depending only on $n$ such that,
	$$C_1\geq C(n)\left(\frac{\vol(M)}{\int_0^{d_M} (\sqrt{-1/K}\sinh{\sqrt{-K}r})^{n-1}dr}\right)^{n+1},$$
where we use the convention that $(\sqrt{-1 / K} \sinh \sqrt{-K} r)$ is interpreted as $r$ if $\mathrm{K}=0$ and as $\sqrt{1 / \mathrm{K}} \sin (\mathrm{K} r)$ if $\mathrm{K}>0$.
\end{theorem}

This provides an alternative way of relating the Gromov norm with the harmonic norm in terms of $n$, $K_p$, a lower bound on volume and an upper bound on diameter.

\def\cprime{$'$}
\providecommand{\bysame}{\leavevmode\hbox to3em{\hrulefill}\thinspace}
\providecommand{\MR}{\relax\ifhmode\unskip\space\fi MR }
\providecommand{\MRhref}[2]{%
  \href{http://www.ams.org/mathscinet-getitem?mr=#1}{#2}
}
\providecommand{\href}[2]{#2}

\end{document}